\documentclass[12pt]{amsart}

\usepackage[pdfauthor   = {Lukas\ Kuehne},
            pdftitle    = {The\ universality\ of\ the\ resonance\ arrangement\ and\ its\ Betti\ numbers},
            pdfsubject  = {},
            pdfkeywords = {matroids;\ resonance arrangement;\ all-subsets arrangement;\ maximal unbalanced families;\ characteristic polynomial;\ Betti numbers},
            bookmarks=true,
            bookmarksopen=true,
            pagebackref=true,
            hyperindex=true,
            colorlinks=true,
            linkcolor=blue,
            citecolor=blue,
            filecolor=blue,
            urlcolor=blue,
            ]{hyperref}

\usepackage[utf8]{inputenc}
\usepackage[T1]{fontenc}
\usepackage[a4paper, includeheadfoot, margin=2.81cm, bottom=1cm]{geometry}                
\usepackage{times}
\usepackage{mathrsfs}
\usepackage{mathtools}
\usepackage{latexsym}
\usepackage{amssymb}
\usepackage{amsthm}
\usepackage{epsfig}
\usepackage{colortbl}
\usepackage{fancyvrb}
\usepackage{graphicx}
\usepackage{subcaption}
\usepackage[font=small,labelfont=bf]{caption}
\usepackage[dvipsnames,table]{xcolor}
\usepackage{accents} 
\usepackage{enumerate}
\usepackage{blkarray, bigstrut}
\usepackage{xparse}

\usepackage{wrapfig}
\usepackage{tikz}
\usetikzlibrary{shapes,arrows,matrix,backgrounds,positioning,plotmarks,calc,patterns,
decorations.shapes,
decorations.fractals,
decorations.markings,
decorations.pathreplacing,
decorations.pathmorphing,
decorations.text
}
\usepackage[colorinlistoftodos,shadow]{todonotes}
\usepackage{bm}
\usepackage{multirow}
\usepackage{mdwlist}
\usepackage{stmaryrd}
\usepackage{mathdots} 
\usepackage{arydshln}
\usepackage[toc,page]{appendix}
\usepackage{float}

\usepackage{bigfoot}

\usepackage[sort&compress,capitalise]{cleveref}

\crefname{exmp}{Example}{Examples}
\crefname{prop}{Proposition}{Propositions}

\usepackage[linesnumbered,commentsnumbered,ruled,vlined]{algorithm2e}

\SetCommentSty{mycommfont}

\newtheoremstyle{mytheoremstyle} 
    {5pt}                    
    {5pt}                    
    {\itshape}                   
    {\parindent}                           
    {\bf}                   
    {.}                          
    {.5em}                       
    {}  

\theoremstyle{mytheoremstyle}

\newtheorem{theorem}{Theorem}[section]

\newtheorem{lemm}[theorem]{Lemma}
\newtheorem{prop}[theorem]{Proposition}

\makeatletter
\@addtoreset{claims}{section}
\makeatother

\newtheoremstyle{mytdefintionstyle} 
    {5pt}                    
    {5pt}                    
    {\rm}                   
    {\parindent}                           
    {\bf}                   
    {.}                          
    {.5em}                       
    {}  

\theoremstyle{remark}
\newtheorem{rmrk}[theorem]{Remark}

\theoremstyle{mytdefintionstyle}
\newtheorem{defn}[theorem]{Definition}
\newtheorem{exmp}[theorem]{Example}

\newtheoremstyle{exmp_contd}
    {5pt}                    
    {5pt}                    
    {\rm}                   
    {\parindent}                           
    {\bf}                   
    {.}                          
    {.5em}                       
    {\thmname{#1}\ \thmnumber{ #2}\thmnote{#3}\ (continued)}  
\theoremstyle{exmp_contd}

\DeclareMathOperator{\codim}{codim}

\newcommand\A{\mathcal{A}}
\newcommand\C{\mathbb{C}}

\newcommand\B{\mathcal{B}}

\newcommand\F{\mathbb{F}}

\newcommand\bF{\mathbb{F}}
\newcommand\I{\mathcal{I}}

\renewcommand\P{\mathcal{P}}
\newcommand{\Q}{\mathbb{Q}}

\newcommand\R{\mathbb{R}}

\newcommand{\Z}{\mathbb{Z}}
\newcommand\N{\mathbb{N}}
\renewcommand\phi{\varphi}

\definecolor{darkgray}{rgb}{0.3,0.3,0.3}
\definecolor{LightGray}{gray}{0.9}

\newcommand{\topstrut}[1][1.2ex]{\setlength\bigstrutjot{#1}{\bigstrut[t]}}
\newcommand{\botstrut}[1][0.9ex]{\setlength\bigstrutjot{#1}{\bigstrut[b]}}

\DeclareDocumentEnvironment{bheadmatrix}{O{}O{*{10}{c}}}{%
	\begin{blockarray}{#2}
		#1 \\[-0.8ex]
		\begin{block}{[#2]}
			\topstrut}%
		{%
			\botstrut\\
		\end{block}
	\end{blockarray}
}%

\definecolor{darkgreen}{rgb}{0.008,0.617,0.067}
\definecolor{brown}{rgb}{0.6,0.4,0.2}

\newif\ifjournalversion

\DeclareDocumentEnvironment{bheadmatrix}{O{}O{*{10}{c}}}{%
	\begin{blockarray}{#2}
		#1 \\[-0.8ex]
		\begin{block}{[#2]}
			\topstrut}%
		{%
			\botstrut\\
		\end{block}
	\end{blockarray}
}%

\author{Lukas K\"uhne}
\address{Einstein Institute of Mathematics, The Hebrew University of Jerusalem, Giv’at Ram, Jerusalem, 91904, Israel}
\address{Max Planck Institute for Mathematics in the Sciences, Inselstr. 22, 04103, Leipzig, Germany}
\email{\href{mailto:Lukas Kuehne<lukas.kuhne@mis.mpg.de>}{lukas.kuhne@mis.mpg.de}}

\begin{document}

\title{The universality of the resonance arrangement and its Betti numbers}
\begin{abstract}

The resonance arrangement $\A_n$ is the arrangement of hyperplanes which has all non-zero $0/1$-vectors in $\R^n$ as normal vectors.
It is the adjoint of the Braid arrangement and is also called the all-subsets arrangement.
The first result of this article shows that any rational hyperplane arrangement is the minor of some large enough resonance arrangement.

Its chambers appear as regions of polynomiality in algebraic geometry, as generalized retarded functions in mathematical physics and as maximal unbalanced families that have applications in economics.
One way to compute the number of chambers of any real arrangement is through the coefficients of its characteristic polynomial which are called Betti numbers.
We show that the Betti numbers of the resonance arrangement are determined by a fixed combination of Stirling numbers of the second kind.
Lastly, we develop exact formulas for the first two non-trivial Betti numbers of the resonance arrangement.

\end{abstract}

\thanks{L.K. was supported by ERC StG 716424 - CASe, a Minerva fellowship of the Max-Planck-Society and the Studienstiftung des deutschen Volkes.}

\maketitle


\section{Introduction}

\subsection{The Resonance Arrangement}
The main object considered in this article is the resonance arrangement:
\begin{defn}
	For a fixed integer $n\ge1$ we define the hyperplane arrangement $\A_n$ as the \emph{resonance arrangement} in $\R^n$ by setting
	$
	\A_n:=\{H_I\mid \emptyset \neq I\subseteq [n]\},
	$
	where the hyperplanes~$H_I$ are defined by
	$
	H_I := \left\{ \sum_{i\in I}x_i = 0 \right\}.
	$
\end{defn}

\begin{figure}[h]
	\includegraphics[width=.3\linewidth]{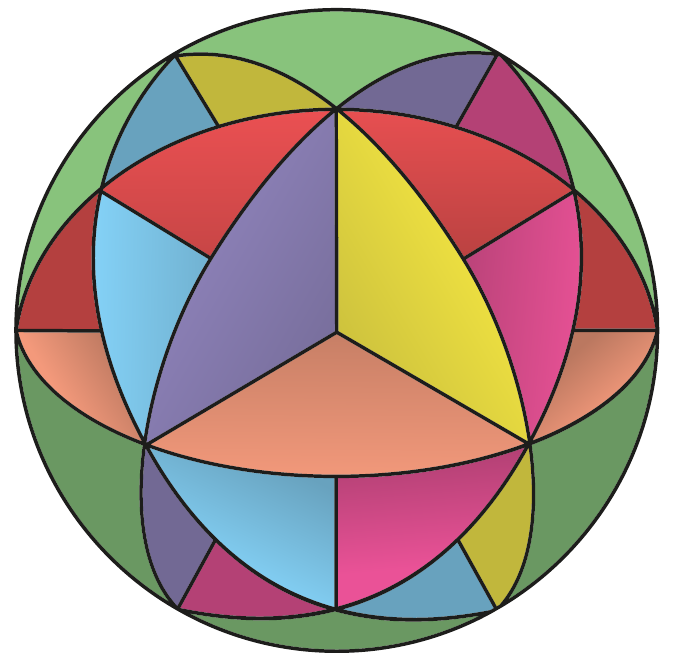}
	\caption{The resonance arrangement $\A_3$ projected onto the hyperplane $H_{\{1,2,3\}}$.
	There are $16$ chambers visible and another $16$ antipodal chambers hidden. Thus, $\A_3$ has $32$ chambers in total.}
	\label{fig:A3}
\end{figure}

The term resonance arrangement was coined by Shadrin, Shapiro, and Vainshtein in their study of double Hurwitz numbers stemming from algebraic geometry~\cite{SSV08}.
Billera, Billey, Rhoades, and Tewari proved that the product of the defining linear equations of~$\A_n$ is Schur positive via a so-called Chern phletysm from representation theory~\cite{BBT18,BRT19}.
Recently, Gutekunst, M{\'e}sz{\'a}ros, and Petersen established a connection between the resonance arrangement and the type $A$ root polytope~\cite{GMP19}.

The arrangement $\A_n$ is also the \emph{adjoint of the braid arrangement}~\cite[Section 6.3.12]{AM17}.
It was studied under this name by Liu, Norledge, and Ocneanu in its relation to mathematical physics~\cite{LNO19}.
The relevance of the resonance arrangement in physics was also demonstrated by Early in his work on so-called \emph{plates}, cf.~\cite{Ear17}.

In earlier work, the arrangement $\A_n$ was called \emph{(restricted) all-subsets arrangement} by Kamiya, Takemura, and Terao who established its relevance for applications in psychometrics and economics~\cite{KTT11,KTT12}.

A first contribution of this article is a universality result of the resonance arrangement for rational hyperplane arrangements:

\begin{theorem}\label{thm:universality}
	Let $\B$ be any hyperplane arrangement defined over $\Q$.
	Then $\B$ is a minor of $\A_n$ for some large enough $n$, that is $\B$ arises from $\A_n$ after a suitable sequence of restriction and contraction steps.
	Equivalently, any matroid that is representable over $\Q$ is a minor of the matroid underlying $\A_n$ for some large enough $n$.
\end{theorem}

The proof is constructive and the size of the required $\A_n$ depends on the size of the entries in an integral representation of $\B$.

\subsection{Chambers of $\A_n$}

The \emph{chambers of $\A_n$} are the connected components of the complement of the hyperplanes in $\A_n$ within $\R^n$.
We denote by $R_n$ the number of chambers of the arrangement $\A_n$.
The arrangement $\A_3$ for instance has $32$ chambers as shown in~\Cref{fig:A3}.

These chambers appear in various contexts, such as quantum field theory where these regions correspond to generalized retarded functions~\cite{Eva95}.
Cavalieri, Johnson, and Markwig proved that the chambers of $\A_n$ are the domains of polynomiality of the double Hurwitz number~\cite{CJM11}.
Subsequently, Gendron and Tahari demonstrated the significance of the chambers of the resonance arrangement in geometric topology~\cite{GT20}.

Billera, Tatch Moore, Dufort Moraites, Wang, and Williams observed that the chambers of $\A_n$ are also in bijection with \emph{maximal unbalanced families} of order $n+1$.
These are systems of subsets of $[n+1] $ that are maximal under inclusion such that no convex combination of their characteristic functions is constant~\cite{BTDWW12}.
Equivalently, the convex hull of their characteristic functions viewed in the $n+1$-dimensional hypercube does not meet the main diagonal.
Such families were independently studied by Bj\"orner as \emph{positive sum systems}~\cite{Bjo15}.

The values of $R_n$ are only known for $n\le 8$ and are given in~\Cref{tab:rn}, cf. also~\cite[\href{https://oeis.org/A034997}{A034997}]{OEIS}.
There is no exact formula known for $R_n$.
The work of Odlyzko and Zuev~\cite{Odl88,Zue92} together with the recent one by Gutekunst, M{\'e}sz{\'a}ros, and Petersen~\cite{GMP19} gives the bounds
\begin{equation}\label{eq_rn}
	n^2 -10n^2/\ln(n)-n+\log_2(n+1)<\log_2(R_n) < n^2-1,
\end{equation}
which in turn yields the asymptotic behavior $\log_2 (R_n) \sim n^2$.
Deza, Pournin, and Rakotonarivo obtained the improved upper bound of $\log_2(R_n) < n^2-3n+2+\log_2(2n+8)$~\cite{DPR}.

Due to a theorem of Zaslavsky the number of chambers of any arrangement over $\R$ equals the sum of all Betti numbers of the arrangement~\cite{Zas75}.
The Betti numbers can be defined via the characteristic polynomial of an arrangement:

\begin{defn}
	For any arrangement of hyperplanes $\A$ in $\F^n$ for any field $\F$ its \emph{characteristic polynomial}~$\chi(\A;t)$ is defined to be
	\[
	\chi(\A;t) := \sum_{S\subseteq \A} (-1)^{|S|}t^{r(\A)-r(S)},
	\]
	where for any subset $S\subseteq \A$ we set $r(S):=\codim \cap_{H\in S} H$.
	The absolute value of the coefficient of $t^{n-i}$ in the characteristic polynomial $\chi(\A;t)$ is called $i$-th \emph{Betti number}.
	One always has $b_0(\A)=1$ and $b_1(\A)=|\A|$.
\end{defn}

In the case of a complex arrangement of hyperplanes, the Betti numbers coincide with the topological Betti numbers of the complement of the arrangement $\C^n \setminus (\cup_{H\in\A}H)$ with coefficients in $\Q$, cf.~\cite[Chapter 5]{OT92} for an overview of the topological study of arrangement complements.

A formula for $\chi(\A_n;t)$ would also yield a formula for $R_n$.
Unfortunately, there is also no such formula known for $\chi(\A_n;t)$.
In fact, the polynomial $\chi(\A_n;t)$ itself is only known for $n\le 7$ as computed in~\cite{KTT11}.

The next result of this article proves that the Betti numbers $b_i(\A_n)$ for any fixed $i>0$ can be computed for all $n>0$ from a fixed finite combination of \emph{Stirling numbers of the second kind} $S(n,k)$ which count the number of partitions of $n$ labeled objects into $k$ non-empty blocks.
The proof is based on Brylawski's broken  circuit complex~\cite{Bry77}.

\begin{theorem}\label{thm:betti_stirling}
	There exist some positive integers $c_{i,k}$ for all $i\ge 0$ and $i+1\le k \le 2^i$ such that for all $n\ge 1$,
	\[
	b_i(\A_n) = \sum_{k=1}^{2^i} c_{i,k} S(n+1,k).
	\]
	Moreover, the constants $c_{i,k}$ are bounded by $c_{i,k} \leq  \binom{2^i-1}{k-1} \frac{(k-1)!}{i!}$.
\end{theorem}
The first two trivial cases of this theorem are 
\[
	b_0(A_n) = S(n+1,1), \quad b_1(\A_n) = S(n+1,2).
\]
One can obtain exact formulas for the higher Betti numbers $b_i(\A_n)$ from \Cref{thm:betti_stirling} if one knows $b_i(\A_n)$ for all $1\le n \le 2^i$ since the matrix of Stirling numbers $(S(n,k))_{n,k=1,\dots, 2^i}$ is invertible.
Unfortunately, this already fails for $b_3(\A_n)$ since $\chi(\A_n;t)$ is only known for $n\le 7$.

Combining the upper bound on the constants $c_{i,k}$ given in~\Cref{thm:betti_stirling} with the formula for the Stirling numbers given in~\eqref{eq:stirling} yields the upper bound $b_{i}(\A_n)<\frac{2^{in}}{i!}$ for $i,n\ge 1$.
Summing up these bounds for $i=0,1,\dots,n$ we obtain for $n>1$
\[
	\log_2(R_n) < n^2-n+1.
\]

Analyzing the triangles in the broken circuit in detail we obtain exact formulas for the first two non-trivial coefficients of $\chi(\A_n,t)$, namely $b_2(\A_n)$ and $b_3(\A_n)$,  in terms of Stirling numbers of the second kind.
That is, we determine the exact constants $c_{2,k}$ and $c_{3,k}$ for all relevant $k$.
The resulting values of $b_2(\A_n)$ and $b_3(\A_n)$ are displayed in~\Cref{tab:rn}.

\begin{theorem}\label{thm:betti}
	For any $n\ge 1$ it holds that
	\begin{align*}
		(i)\quad b_2(\A_n) = &  2S(n+1,3) + 3S(n+1,4), \\
		=&\frac{1}{2} (4^n - 3^n - 2^n + 1) \mbox{ and}	\\
		(ii) \quad b_3(\A_n) =  & 9S(n+1,4)+80S(n+1,5) + 345 S(n+1,6) \\
		&+ 840 S(n+1,7) +840S(n+1,8) ,\\
		=&\frac{1}{4!} (4\cdot8^n -15\cdot 6^n +15\cdot 5^n - 14 \cdot 4^n + 18 \cdot 3^n - 7\cdot 2^n-1).
	\end{align*}
\end{theorem}

\begin{exmp}
	Using~\Cref{thm:betti} we can compute $\chi(\A_3;t)$ as
	\[
		\chi(\A_3;t) = t^3 - 7t^2 + 15t-9.
	\]
	Thus, the above mentioned result by Zaslavsky again yields $R_3=1+7+15+9=32$.
\end{exmp}

\begin{rmrk}
	The formula for $b_2(\A_n)$ in~\Cref{thm:betti} $(i)$ was also found earlier by Billera (personal communication).
\end{rmrk}

\begin{table}
	\[\def\arraystretch{1.5}{
		\begin{array}{c| c c c c c c c c c}
		n & 1 & 2 & 3 & 4 & 5 & 6 & 7 & 8 & 9 \\
		\hline
		\hline
		b_1(\A_n)=|\A_n| & 1 & 3 & 7 & 15 & 31 & 63 & 127 & 255 & 511\\
		b_2(\A_n) & 0 & 2& 15& 80& 375& 1652& 7035& 29360& 120975\\
		b_3(\A_n) & 0 & 0 & 9 & 170 & 2130 & 22435 & 215439 & 1957200 & 17153460\\
		b_4(\A_n) & 0 & 0 &  0 & 104 & 5270 & 159460 & 3831835 & ? & ?\\
		\hline
		R_n & 2 & 6 & 32 & 370 & 11292 & 1066044 & 347326352 & 419172756930 & ?
		\end{array}}
	\]
	\caption{The known values of $b_i(\A_n)$ for $1\le i \le 4$ and $R_n$ which is the number of chambers of $\A_n$.
	The values for $b_2(\A_n)$ and $b_3(\A_n)$ were computed using~\Cref{thm:betti}.}
	\label{tab:rn}
\end{table}

This article is organized as follows.
After reviewing necessary definitions of matroids and their minors in~\Cref{sec:matroids} we will prove~\Cref{thm:universality} in~\Cref{sec:universality}.
Subsequently, we state the necessary facts on broken circuit complexes in~\Cref{sec:bc} and prove~\Cref{thm:betti_stirling} in~\Cref{sec:stirling}.
Lastly, we give the proof of~\Cref{thm:betti} in~\Cref{sec:b2,sec:b3}.

\section*{Acknowledgments}

I would like to thank Karim Adiprasito for his mentorship and for introducing me to the topic of resonance arrangements.
Furthermore, I am grateful to Louis Billera, Michael Joswig, and Jos\'e Alejandro Samper for helpful conversations and feedback on earlier version of this manuscript.
Last but not least, I am indebted to the graphics department of the Max Planck Institute for Mathematics in the Sciences for helping me to create \Cref{fig:A3}.

\section{Matroids and their Minors}\label{sec:matroids}

In this section we review some basics of matroids and their minors.
Details can be found in~\cite{Oxl11}.

\begin{defn}
	A \emph{matroid} $M$ is a pair $(E,\I)$ where $E$ is a finite \emph{ground set} and $\I$ is a non-empty family of subsets of $E$, called \emph{independent sets} such that
	\begin{enumerate}[(i)]
		\item for all $A'\subseteq A \subseteq E$ if $A\in \I$ then $A'\in\I$ and
		\item if $A,B\in \I$ with $  |A|>|B|$ then there exists $a\in A\setminus B$ such that $B\cup\{a\}\in \I$.
	\end{enumerate}	
\end{defn}

Given some set finite set $E$ and an $r\times E$-matrix $A$ with entries in some field $\F$ we obtain a matroid $M(A)$ on the ground set $E$ whose independent sets are the columns of~$A$ that are linear independent.
A matroid $M$ is called \emph{representable} over a field $\F$ if there exists an $r\times E$-matrix $A$ such that $M=M(A)$.

An arrangement of hyperplanes $\A$ also gives rise to a matroid by writing the coefficients of a linear equation for each $H\in\A$ as columns in a matrix and applying the above construction.
Similarly, we also get a matroid $M(\A)$ underlying an arrangement $\A$ with ground set $\A$ whose independent set are precisely those whose hyperplanes intersect with codimension equal to the cardinality of the subset.

\begin{defn}
	Let $M=(E,\I)$ be a matroid and $S\subseteq E$. Then one defines:
	\begin{enumerate}
		\item The \emph{restriction} of $M$ to $S$, denoted $M|S$, is the matroid on the ground set $S$ with independent sets $\{ I\in \I\mid I\subseteq S\}$.
		\item Assume that $S$ is independent in $M$.
		Then, the \emph{contraction} of $M$ by $S$, denoted $M/S$, is the matroid on the ground set $E\setminus S$ with independent sets $\{ I \subseteq E\setminus S\mid I\cup S\in \I \}$.
	\end{enumerate}
	A matroid $N$ is called a \emph{minor} of $M$ if $N$ arises from $M$ after a finite sequence of restrictions and contractions.
\end{defn}

Minors play a central role in the theory of matroids.
For instance, Geelen, Gerards and Whittle announced a proof of Rota's conjecture which asserts that matroid representability over a finite field can be characterized by a finite list of excluded minors~\cite{GGW14}.

The restriction of a representable matroid to some subset $S$ is again representable by the same matrix after removing the columns that are not in $S$.
The following lemma establishes a similar connection for contractions of representable matroids.
This also motivates the term minor of a matroid as it corresponds to a minor of a matrix in the representable case.

\begin{lemm}\cite[Proposition 3.2.6]{Oxl11}\label{lem:oxley}
	Let $E$ be some finite set and $A$ an $r\times E$ matrix over a field $\F$.
	Suppose $e\in E$ is the label of a non-zero column of $A$.
	Let $A'$ be the matrix arising from $A$ through row operations by pivoting on some non-zero element in the column $e$.
	Let $A'/e$ be the matrix $A'$ where one removes the row and column containing the unique non-zero entry in the column $e$.
	Then,
	\[
		M(A)/e =M(A')/e=  M(A'/e).
	\]
\end{lemm}

\section{Universality of the Resonance Arrangement}\label{sec:universality}

Let $M$ be a matroid of rank $r$ and size $n$ that is representable over $\Q$.
Thus after scaling, we can assume that there is a $r\times n$ matrix $A$ with entries in $\Z$ that represents $M$.
Let $a_1,\dots,a_n\in \Z^r$ be the column vectors of the matrix $A$.
Expressing each vector~$a_i$ for $1\le i \le n$ as a sum of positive and negative characteristic vectors yields
\begin{equation}\label{eq:ai}
a_i = \sum_{j=1}^{m_i^+} \chi_{P_j^{i}}  - \sum_{k=1}^{m_i^-} \chi_{N_k^{i}},
\end{equation}
for some $m_i^+,m_i^-\in \N$ and $P_j^i,N_k^i\subseteq [n]$ for all $1\le j \le m_i^+$ and $1\le k \le m_i^-$.

We work in the extended vector space
\[
\Q^N \coloneqq \Q^r \times \Q^{m_1^-}\times \Q^{m_1^+}\times \Q^{m_1^{+}} \times \dots \times \Q^{m_n^-}\times \Q^{m_n^+}\times \Q^{m_n^{+}} ,
\]
for some appropriate $N\in\N$.
Hence, the vectors $a_1,\dots,a_n$ naturally live in the first factor~$\Q^r$ of $\Q^N$.
We fix the standard basis of $\Q^N$ as
\[
e_1,\dots,e_r,e_1^{1,-},\dots,e_{m_1^-}^{1,-},e_1^{1,+},\dots,e_{m_1^+}^{1,+},e_1^{1,++},\dots,e_{m_1^+}^{1,++},\dots.
\]

Now, we describe a construction which will be used in the proof in~\Cref{thm:universality}.
To this end, we define $0/1$-vectors $v_1,\dots,v_n$ which will eventually represent the matroid $M$ after contracting several other $0/1$-vectors.
We define for each $1\le i \le n$:
\begin{align*}
v_i \coloneqq & \sum_{j=1}^{m_i^+} e_j^{i,++}+  \sum_{k=1}^{m_i^-} e_k^{i,-}, \\
r_k^{i,-} \coloneqq & \chi_{N_k^{i}} + e_k^{i,-} \mbox{ for }1\le k \le m_i^-, \\
r_j^{i,+} \coloneqq & \chi_{P_j^{i}} + e_j^{i,+} \mbox{ for }1\le j \le m_i^+,\\
r_j^{i,++} \coloneqq & e_j^{i,+}  + e_j^{i,++}\mbox{ for }1\le j \le m_i^+. 
\end{align*}

We collect these vectors in the sets $V\coloneqq \{ v_1,\dots,v_n\}$ and
\[
R\coloneqq \{ r_k^{i,-}, r_j^{i,+} ,r_j^{i,++} \mid 1\le i \le n, 1\le k \le m_i^- \mbox{ and }1\le j \le m_i^+ \}.
\]


\begin{exmp}
	Consider the vectors $a_1\coloneqq (1,-2,-1)^T$ and $a_2\coloneqq (-1,0,-1)^T$ in $\Z^3$.
	They can be expressed as $a_1   = \chi_{\{1\}} - \chi_{\{2,3\}} - \chi_{\{2\}}$ and $a_2  = - \chi_{\{ 1,3 \}}$.
	
	Thus, $m_1^{-}=2, m_1^{+}=1,m_2^{-}=1$, and $m_2^{+}=0$.
	The above construction yields the following column vectors in $\Q^8$ depicted in the left matrix below.
	The matrix on the right arises from the one on the left after suitable row operations as described below in the proof of~\Cref{thm:universality}.
	\[
		\begin{array}{r}
		v_1\quad r_1^{1,-}\ r_2^{1,-}\ r_1^{1,+}\, r_1^{1,++}\! v_2\ \quad r_1^{2,-} \\
		\left[\setlength{\arraycolsep}{8pt}\def\arraystretch{1.2}
		\begin{array}{>{\columncolor{gray!20}}c c >{\columncolor{gray!20}}cc>{\columncolor{gray!20}}cc>{\columncolor{gray!20}}c}
		\rowcolor{gray!20}
		0 &  0 & 0 & 1 & 0 & 0 & 1 \\ 
		0 &  1 & 1 & 0 & 0 & 0 & 0 \\ \rowcolor{gray!20}
		0 &  1 & 0 & 0 & 0 & 0 & 1 \\ 
		1 &  1 & 0 & 0 & 0 & 0 & 0 \\ \rowcolor{gray!20}
		1 &  0 & 1 & 0 & 0 & 0 & 0 \\ 
		0 &  0 & 0 & 1 & 1 & 0 & 0 \\ \rowcolor{gray!20}
		1 &  0 & 0 & 0 & 1 & 0 & 0 \\ 
		0 &  0 & 0 & 0 & 0 & 1 & 1 
		\end{array}
		\right]
		\end{array}
		\leadsto
		\begin{array}{r}
		\phantom{r_1^{1,+}}\\
		\left[\setlength{\arraycolsep}{8pt}\def\arraystretch{1.2}
		\begin{array}{>{\columncolor{gray!20}}c c >{\columncolor{gray!20}}cc>{\columncolor{gray!20}}cc>{\columncolor{gray!20}}c}
		\rowcolor{gray!20}
		1 &  0 & 0 & 0 & 0 & -1 & 0 \\ 
	  -2 &  0 & 0 & 0 & 0 & 0 & 0 \\ \rowcolor{gray!20}
	  -1 &  0 & 0 & 0 & 0 & -1 & 0 \\ 
		1 &  1 & 0 & 0 & 0 & 0 & 0 \\ \rowcolor{gray!20}
		1 &  0 & 1 & 0 & 0 & 0 & 0 \\ 
		-1 &  0 & 0 & 1 & 0 & 0 & 0 \\ \rowcolor{gray!20}
		1 &  0 & 0 & 0 & 1 & 0 & 0 \\ 
		0 &  0 & 0 & 0 & 0 & 1 & 1 
		\end{array}
		\right]
		\end{array}.
	\]
	All columns apart from $v_1,v_2$ became standard basis vectors and removing those columns together with all rows apart from the first three yields the matrix with columns $a_1,a_2$.
\end{exmp}

\begin{proof}[Proof of~\Cref{thm:universality}]
	
	Assembling the vectors in $R$ and $V$ to a matrix yields:

	\begin{equation}\label{eq:large_mat}
	\begin{array}{c r}
	& \begingroup\setlength{\arraycolsep}{8pt}
	v_1  \; \quad r^{1,-}_\ast  \quad r^{1,+}_\ast  \quad r^{1,++}_\ast  \quad v_2 \qquad r^{2,-}_\ast \quad r^{2,+}_\ast \quad r^{2,++}_\ast \quad \cdots \quad
	\endgroup\\[2pt]
	\begin{array}{ c} 
	\phantom{a} \\ 
	e_\ast\\
	\phantom{\vdots} \\
	\phantom{a} \\ 
	e_\ast^{1,-}\\
	\phantom{\vdots} \\
	\phantom{a} \\ 
	e_\ast^{1,+}\\
	\phantom{\vdots} \\
	\phantom{a} \\ 
	e_\ast^{1,++}\\
	\phantom{a} \\
	\phantom{\vdots} \\
	e_\ast^{2,-}\\
	\phantom{\vdots} \\
	\phantom{a} \\ 
	e_\ast^{2,+}\\
	\phantom{\vdots} \\
	\phantom{a} \\ 
	\phantom{a} \\ 
	e_\ast^{2,++}\\
	\phantom{a} \\
	\vdots
	\end{array} & 
	\left[\setlength{\arraycolsep}{8pt}\begin{array}{c ;{.4pt/1pt} c ;{.4pt/1pt} c;{.4pt/1pt} c  ;{2pt/2pt}c ;{.4pt/1pt} c;{.4pt/1pt} c;{.4pt/1pt} c  ;{2pt/2pt}c }
	0 		  & 				&			  		&					& 0 		    & 				 &			  		&	& \\
	\vdots & 	\ast	&	\ast	 		&		0 		  &   \vdots   & 	\ast	 &	\ast	 	  &	0& \cdots\\ 	 	
	0 		  & 				&			  		&		   			&0 		  		& 				 &			  		&	& \\\hdashline[2pt/2pt]
	1		  & 				&			  		&					&0				&				 &					&& \\
	\vdots &I_{m_1^{-}}&      0	 		&		0   		 &\vdots	&0				 &0				  &0&\cdots \\ 	 	
	1 		  & 				&			  		&					&0				&				 &					&& \\	\hdashline[.4pt/1pt]		
	0 		  & 				&			  		&					&0				&				 &					&& \\
	\vdots & 	0       	 &I_{m_1^{+}}&I_{m_1^{+}} &\vdots	  &0			  &0				&0& \cdots\\ 	 	
	0 		  & 				&			  		&					&0				&				 &					&& \\\hdashline[.4pt/1pt]		
	1 		  & 				&			  		&					&0				&				 &					&& \\
	\vdots & 	0       	&	0	    	   &I_{m_1^{+}} &\vdots		&0				&0				  &0& \cdots\\	 	
	1		  & 				&			  		&					&0				&				  &					&	& \\ \hdashline[2pt/2pt]		
	0		  & 				&			  		&					&1				&				 &					&& \\
	\vdots &0				&      0	 		&		0   	  &\vdots	  &I_{m_2^{-}}&0			   &0& \cdots\\ 	 	
	0 		  & 				&			  		&					&1				&				 &					&& \\	\hdashline[.4pt/1pt]		
	0 		  & 				&			  		&					&0				&				 &					&& \\
	\vdots & 	0       	&0				   &0   			&\vdots	  &0			  &I_{m_2^{+}}&I_{m_2^{+}}& \cdots\\ 	 	
	0 		  & 				&			  		&					&0				&				 &					&& \\\hdashline[.4pt/1pt]
	0 		  & 				&			  		&					&1				&				 &					&& \\
	\vdots & 	0       	&	0	    	   &0				 &\vdots		&0				&0				  &I_{m_2^{+}}& \cdots\\	 	
	0		  & 				&			  		&					&1				&				  &					&		& \\[2pt]\hdashline[2pt/2pt]		
	\vdots& \vdots		&\vdots	  		&\vdots		   &\vdots		&\vdots		  &\vdots	 & \vdots& \ddots
	\end{array}\right]
	\end{array}.
	\end{equation}
	Now, we perform row operations on the matrix in~\eqref{eq:large_mat} to ensure that all columns corresponding to vectors in $R$ are standard basis vectors.
	To this end, we apply the following steps for all $1\le i \le n$:
	\begin{enumerate}
		\item We pivot on the entry in row $e_k^{i,-}$ and column $r_k^{i,-}$ for each $1\le k \le m_i^{-}$.
		\item Lastly, we pivot on the entry in row $e_j^{i,s}$ and column $r_j^{i,s}$ for each $1\le j \le m_i^{+}$ and each $s\in \{+,++\}$.
	\end{enumerate}
	By construction and~\Cref{eq:ai}, this procedure yields the following matrix:
	\begin{equation}\label{eq_mat_after_pivot}
	\left[\setlength{\arraycolsep}{8pt}\begin{array}{c ;{.4pt/1pt} c ;{.4pt/1pt} c;{.4pt/1pt} c  ;{2pt/2pt}c ;{.4pt/1pt} c;{.4pt/1pt} c;{.4pt/1pt} c  ;{2pt/2pt}c }
	& 				&			  		&					&  		    	& 				 &			  		&	& \\
	a_1 	& 	0			&	0		 		&		0 		  &   a_2   	& 	0			 &	0		 	  &	0& \cdots\\ 	 	
	& 				&			  		&		   			& 		  		& 				 &			  		&	& \\\hdashline[2pt/2pt]
	1		  & 				&			  		&					&0				&				 &					&& \\
	\vdots &I_{m_1^{-}}&      0	 		&		0   		 &\vdots	&0				 &0				  &0&\cdots \\ 	 	
	1 		  & 				&			  		&					&0				&				 &					&& \\	\hdashline[.4pt/1pt]		
	-1 		  & 				&			  		&					&0				&				 &					&& \\
	\vdots & 	0       	 &I_{m_1^{+}}&0					&\vdots	  &0			  &0				&0& \cdots\\ 	 	
	-1 		  & 				&			  		&					&0				&				 &					&& \\\hdashline[.4pt/1pt]		
	1 		  & 				&			  		&					&0				&				 &					&& \\
	\vdots & 	0       	&	0	    	   &I_{m_1^{+}} &\vdots		&0				&0				  &0& \cdots\\	 	
	1		  & 				&			  		&					&0				&				  &					&	& \\ \hdashline[2pt/2pt]		
	0		  & 				&			  		&					&1				&				 &					&& \\
	\vdots &0				&      0	 		&		0   	  &\vdots	  &I_{m_2^{-}}&0			   &0& \cdots\\ 	 	
	0 		  & 				&			  		&					&1				&				 &					&& \\	\hdashline[.4pt/1pt]		
	0 		  & 				&			  		&					&-1				&				 &					&& \\
	\vdots & 	0       	&0				   &0   			&\vdots	  &0			  &I_{m_2^{+}}&0& \cdots\\ 	 	
	0 		  & 				&			  		&					&-1				&				 &					&& \\\hdashline[.4pt/1pt]
	0 		  & 				&			  		&					&1				&				 &					&& \\
	\vdots & 	0       	&	0	    	   &0				 &\vdots		&0				&0				  &I_{m_2^{+}}& \cdots\\	 	
	0		  & 				&			  		&					&1				&				  &					&		& \\[2pt]\hdashline[2pt/2pt]		
	\vdots& \vdots		&\vdots	  		&\vdots		   &\vdots		&\vdots		  &\vdots	 & \vdots& \ddots
	\end{array}\right].
	\end{equation}
	Therefore, we obtain the matrix $A$ from the one given in~\Cref{eq_mat_after_pivot} by removing all columns corresponding to vectors in $R$ and all rows apart from the first $r$ ones.
	Hence, \Cref{lem:oxley} implies that the matroid $M$ equals the matroid of the resonance arrangement $\A_N$ restricted to $V\cup R$ and contracted by $R$, that is $M$ is a minor of the matroid of $\A_N$.
\end{proof}

\section{The Broken Circuit Complex}\label{sec:bc}

The \emph{Stirling numbers of the second kind} are denoted by $S(n,k)$ and count the number of ways to partition $n$ labeled objects into $k$ nonempty unlabeled blocks.
We will use the standard formula
\begin{equation}\label{eq:stirling}
	S(n,k) = \frac{1}{k!} \sum_{i=0}^k(-1)^i\binom{k}{i} (k-i)^n.
\end{equation}

A tool to compute the Betti numbers of an arrangement is the broken circuit complex:
\begin{defn}\label{def:broken_circuits}
	Let $\A$ be any arrangement and fix any linear order $<$ on its hyperplanes.
	A \emph{circuit} of $\A$ is a minimally dependent subset.
	A \emph{broken circuit} of $\A$ is a set $C\setminus \{H\}$ where $C$ is a circuit and $H$ is its largest element (in the ordering $<$).
	The \emph{broken circuit complex} $BC(\A)$ is defined by
	\[
		BC(\A)\coloneqq \{ T \subset \A \mid T \mbox{ contains no broken circuit}\}.
	\]
\end{defn}

Its significance lies in the following result:
 \begin{theorem}\cite{Bry77}\label{thm:betti_bc}
	Let $\A$ be any arrangement in a vector space $\bF^n$ for some field $\F$ with a fixed linear order~$<$ on its hyperplanes.
	Then for any $1\le i \le n$ it holds that
	\[
		b_{i}(\A) = f_{i-1}(BC(\A)),
	\]
	where $f_i$ is the $f$-vector of the broken circuit complex.
\end{theorem}

For the rest of the article we will study the broken circuit complex of the resonance arrangement $\A_n$.
Each subset of  $I \subseteq \left[ n \right]$ can be encoded as a binary number $\sum_{i \in I} 2^i$.
This gives rise to a natural ordering of the hyperplanes in $\A_n$ which we will use as to obtain its broken circuit complex.
In the subsequent proofs we will identify a hyperplane $H_A$ with its defining subset $A$ or its corresponding characteristic vector $\chi_A$ if no confusion arises.

\section{Proof of~\Cref{thm:betti_stirling}}\label{sec:stirling}

Throughout this section we use the following notation:
Taking all possible intersections of the sets in an $i$-tuple $(A_1,\dots,A_i)$ of pairwise different non-empty subsets of $[n]$ yields a partition $\pi=\{P_1,\dots,P_k\}$ of $[n+1]$ into $k$ blocks with $i+1\le k\le 2^i$ (the block containing $n+1$ exactly contains all elements of $[n]$ which are not contained in any of the sets $A_j$ for $1\le j \le i$.
We order the blocks in the partition $\pi$ by their binary representation as detailed above; in particular we have $n+1\in P_k$.

We can recover the tuple $(A_1,\dots,A_i)$ from the partition $\pi$ through a map 
\begin{align*}
f:[k-1]&\rightarrow \P([i])\setminus\{\emptyset\},\\
\ell &\mapsto \{ j\in [i]\mid P_\ell \subseteq A_j \},
\end{align*}
Note that such a map is injective since the sets in the  $(A_1,\dots,A_i)$ are assumed to be pairwise different.
We call any injective map $f:[k-1]\rightarrow \P([i])\setminus\{\emptyset\}$ an $(i,k)$-\emph{prototype}.

Conversely, given any partition $\pi=\{P_1,\dots,P_k\}$ of $[n+1]$ and a $(i,k)$-prototype $f$ we obtain an $i$-tuple $(A_1,\dots,A_i)$ which we denote by $A_{f,\pi}$ by setting for $1\le j \le i$
\[
A_j \coloneqq  \bigcup_{\ell \in I^f_j }P_{\ell},
\]
where we define $I^f_j \coloneqq \{\ell \in [k-1]  \mid j\in f(\ell) \}$ for $1\le j \le i$ and call these sets the \emph{building blocks} of $f$.

In total, this construction gives a bijection between $i$-tuples of pairwise different non-empty subsets of $[n]$ and pairs of $(i,k)$-prototypes together with partitions of $[n+1]$ into $k$ blocks with $i+1\le k \le 2^i$.

Now the main observation is the following.
Whether an $i$-tuple $A_{f,\pi}$ is a broken circuit depends only on the prototype $f$ but not on the partition $\pi$:

\begin{prop}\label{prop:prototypes}
	In the above notation, let $f:[k-1]\rightarrow \P([i])\setminus\{\emptyset\}$ be an $(i,k)$-prototype.
	Assume there exists a partition $\pi=\{P_1,\dots,P_k\}$ of $[n+1]$ such that the $i$-tuple $A_{f,\pi}=(A_1,\dots,A_i)$ is a broken circuit of $\A_n$ (in the order induced by the binary representation).
	
	Let $\widetilde{\pi}=\{\widetilde{P_1},\dots,\widetilde{P_k}\}$ be any partition of $[\widetilde{n}+1]$ for some $\widetilde{n}\ge 1$ into $k$ non-empty parts.
	Then the $i$-tuple $A_{f,\widetilde{\pi}}=(\widetilde{A_1},\dots,\widetilde{A_i})$ is also a broken circuit of $\A_{\widetilde{n}}$.
\end{prop}
\begin{proof}
	By assumption, the tuple $A_{f,\pi} = (A_1,\dots,A_i)$ is a broken circuit.
	Thus, there exists some $C \subseteq [n]$ and $\lambda_1,\dots,\lambda_i\in \R^*$ such that
	\begin{equation}\label{eq:bc}
		\sum_{j=1}^i \lambda_j\chi_{A_j} = \chi_{C},
	\end{equation}
	and $A_j < C$ for all $1\le j \le i$.
	
	This implies that $C$ is also a union of the first $k-1$ parts of the partition $\pi$, that is there exists some $I_C\subseteq [k-1]$ such that $C=\bigcup_{\ell \in I_C}P_\ell$.
	Hence, we can rewrite~\Cref{eq:bc} as
	\begin{equation}\label{eq:bc2}
		\sum_{j=1}^i \lambda_j\sum_{\ell \in I_j^f} P_\ell  = \sum_{\ell \in I_C} P_\ell,
	\end{equation}
	Subsequently, the fact $A_j < C$  yields $I_j^f < I_C$ for all $1\le j \le i$ where $I_j^f$ are the building blocks of the prototype $f$ and the order is the one induced by the binary representation of subsets of $[k-1]$.
	
	Now consider the partition $\widetilde{\pi}$ of $[\widetilde{n}+1]$.
	Using the building block $I_C$ of $C$ we can define a corresponding subset of $[\widetilde{n}]$ by setting $\widetilde{C}\coloneqq \bigcup_{\ell \in I_C} \widetilde{P_\ell}$.
	Thus,~\Cref{eq:bc2} implies
	\[
	\sum_{j=1}^i \lambda_j\sum_{\ell \in I_j^f} \widetilde{P_\ell}  = \sum_{\ell \in I_C} \widetilde{P_\ell}.
	\]
	Therefore, the tuple $(\widetilde{A_1},\dots,\widetilde{A_i},\widetilde{C})$ is a circuit of $\A_{\widetilde{n}}$.
	Using the fact $I_j^f < I_C$ we obtain again $\widetilde{A_j}<\widetilde{C}$ for all $1\le j \le i$ which completes the proof that $A_{f,\widetilde{\pi}}$ is a broken circuit in~$\A_{\widetilde{n}}$.
\end{proof}

	In light of~\Cref{prop:prototypes} we can subdivide prototypes into two sets.
	We call those which contain a broken circuit for some partition, and thus for all partitions, \emph{broken} prototypes.
	Otherwise, we call a prototype \emph{functional}.

\begin{proof}[Proof of~\Cref{thm:betti_stirling}]
	As explained above, any $i$-tuple of subsets of $[n]$ can be obtained from an $(i,k)$-prototype and a partition $\pi$ of $[n+1]$ into $k$ blocks with $i+1\le k \le 2^i$.
	\Cref{thm:betti_bc} then implies that we can compute the Betti number $b_i(\A_n)$ for any $i\ge 0$ through functional prototypes and partitions.
	We correct the fact that latter yields ordered tuples unlike the elements  in the broken circuit complex by multiplying the Betti numbers  $b_i(\A_n)$ by $i!$ in the following computation:
	\begin{align*}
		b_i(\A_n) i!=& |\{ X=(A_1,\dots,A_i)\mid A_j\in \P([n])\setminus\{\emptyset\} , A_j\neq A_{j'} \mbox{ for all }j\neq j'\mbox{ and}\\
		&\; \mbox{ $X$ does not contain a broken circuit}  \}|\\ 
		=& \sum_{k=i+1}^{2^i} |\{A_{f,\pi}\mid f \mbox{ functional $(i,k)$-prototype and}\\
		& \qquad\quad \pi \mbox{ partition of $[n+1]$ into $k$ blocks}\}|\\
		=& \sum_{k=i+1}^{2^i} |\{\mbox{functional $(i,k)$-prototypes}\}|S(n+1,k).
	\end{align*}
	This already proves that for each $i\ge 0$ the Betti number $b_i(\A_n)$ can be computed by a combination of Stirling numbers which is independent from $n$. This settles the first claim of the theorem.
	
	For the second claim, note that the above argument shows
	\[
		c_{i,k} = \frac{|\{\mbox{functional $(i,k)$-prototypes}\}|}{i!},
	\]
	for all $i\ge 1$ and $i+1\le k \le 2^i$.
	Bounding the number of functional $(i,k)$-prototypes by the number of all $(i,k)$-prototypes which are merely injective functions $f:[k-1]\rightarrow \P([i])\setminus\{\emptyset\}$ immediately yields for all $i\ge 1$ and $i+1\le k \le 2^i$
	\[
	c_{i,k} \le  \binom{2^i-1}{k-1}\frac{(k-1)!}{i!}.\qedhere
	\]
\end{proof}

\begin{rmrk}
	The above upper bound on $c_{2,2^2}$ and $c_{3,2^3}$ actually agrees with the actual value of these constants given in~\Cref{thm:betti} ($3$ and $840$).
	It can be shown that the given bound on $c_{i,2^i}$ is attained for all $i\ge 1$, that is all $(i,k)$-prototypes are functional.
	For $c_{i,k}$ with $i\ge 1$ and $k<2^i$ the upper bound is not tight in general.
\end{rmrk}

\section{The Betti Number $b_2(\A_n)$}\label{sec:b2}

We compute $b_2(\A_n)$ using \Cref{thm:betti_bc}.

\begin{prop}
	For all $n\ge 1$ it holds that
	\[
		f_1(BC(\A_n))= 2S(n+1,3) + 3S(n+1,4) .
	\]
\end{prop}
\begin{proof}
	The only circuits of $\A_n$ of cardinality three are of the form $\{ H_A,H_B,H_{A\cup B}\}$ where $A,B$ are disjoint subsets of $\left[ n \right]$.
	Hence, the only broken circuits of cardinality two are of the form $\{ H_A,H_B\}$ where $A,B$ are disjoint subsets of $\left[ n \right]$.
	Therefore, we are left with counting subsets of the form  $\{ H_A,H_B\}$ where both $A,B$ are non-empty subsets of $\left[ n \right]$ and $A\cap B \neq \emptyset$.
	
	Assume $A\not \subseteq B$ and $B\not \subseteq A$.
	This case corresponds to a partition of $\left[ n+1 \right]$ into four nontrivial blocks $P_1,P_2,P_3,P_4$ where we assume that $n+1\in P_4$.
	Subsequently, we can choose any $P_i$ with $1\le i \le 3$ to be the intersection and set $A\coloneqq P_j \cup P_i$ and $B\coloneqq P_k\cup P_i$ where $\{j,k \}\coloneqq \{1,2,3\}\setminus \{i\}$.
	Thus, there are $3S(n+1,4)$ many possibilities of that type.
	
	Now assume  $A\not \subseteq B$. The subsets of the form $\{ H_A,H_B\}$ with $A \subseteq B$ corresponds to a partition of $\left[ n+1 \right]$ into three nontrivial blocks $P_1,P_2,P_3$ where we again assume $n+1 \in P_3$.
	In this situation we have the two families $\{H_{P_1},H_{P_1\cup P_2}\}$ and $\{H_{P_2},H_{P_1\cup P_2}\}$ which yields $2S(n+1,3)$ possibilities in total of that type.
\end{proof}

\begin{rmrk}
	In the language of the previous section, the above proof implies that all three $(2,4)$-prototypes are functional whereas only two of the three $(2,3)$-prototypes are functional.
\end{rmrk}

Combining this proposition with \Cref{thm:betti_bc} and Equation~\eqref{eq:stirling} yields a proof of the announced formula for $b_2(\A_n)$:
\begin{proof}[Proof of~\Cref{thm:betti} $(i)$]
	We compute:
	\begin{align*}
		b_2(\A_n) = & 2S(n+1,3) + 3S(n+1,4) \\
		=& \frac{2}{3!}(3^{n+1} -3\cdot 2^{n+1} +3)\frac{3}{4!} + (4^{n+1} -4\cdot 3^{n+1} + 6\cdot 2^{n+1} -4)\\
		=&\frac{1}{2} (4^n - 3^n - 2^n + 1).\qedhere
	\end{align*}
\end{proof}

\section{The Betti Number $b_3(\A_n)$}\label{sec:b3}

To compute $b_3(\A_n)$ we again use the broken circuit complex with the ordering induced by the encoding in binary numbers.
Hence, we need to understand which families $\{H_A,H_B,H_C\}$ form a broken circuit of $\A_n$ where $A,B,C$ are subsets of $\left[ n \right]$ that are pairwise not disjoint.
We use the following result due to Jovovic and Kilibarda:
\begin{theorem}[\cite{JK99}]\label{thm:intersecting_families}
	For any $n\ge1$, the number of families $\{A,B,C\}$ where $A,B,C$ are subsets of $\left[ n \right]$ that are pairwise not disjoint is
	\[
		 \frac{1}{3!}(8^n - 3\cdot6^n + 3\cdot 5^n - 4\cdot  4^n + 3 \cdot 3^n + 2\cdot 2^n - 2).
	\]
	Expanding this numbers as sum of Stirling number of the second kind we obtain the equivalent formula
	\begin{equation}\label{eq:intersection_family_stirling}
		13 S(n+1,4) + 92S(n+1,5) + 360S(n+1,6) + 840S(n+1,7) + 840 S(n+1,8).
	\end{equation}
	We call such families \emph{pairwise intersecting}.
\end{theorem}

As a first step we will classify the circuits of $\A_n$ of cardinality four.
To determine the broken circuits it suffices to consider circuits whose first three elements in the ordering $<$ are pairwise intersecting.
Otherwise, the edges between these elements are already broken circuits and therefore not part of $BC(\A_n)$.
\begin{defn}
	We call a circuit in $\A_n$ \emph{relevant} if the corresponding subsets of $\left[ n \right]$ which are not maximal in the circuit are pairwise intersecting.
\end{defn}

\begin{prop}\label{prop:circuits}
	For $n\ge 1$, a four element family in $\A_n$ is a relevant circuit if and only if it is one of the following types for subsets $A_1,A_3,X\subseteq \left[n\right]$ such that 
	\begin{equation}\label{eq:sets}\tag{$\star$} 
	A_1\cap A_3\neq \emptyset, A_1\setminus A_3\neq \emptyset,A_3\setminus A_1\neq \emptyset \mbox{ and }A_1\cap A_3 \cap X=\emptyset:
	\end{equation}
	\begin{enumerate}[(i)]
		\item $\{H_{A_1},H_{A_3},H_{A_1\triangle A_3},H_{A_1\cup A_3}\}$,
		\item $\{H_{A_1},H_{A_3},H_{A_1\cap A_3},H_{A_1\triangle A_3}\}$,
		\item $\{H_{A_1},H_{A_3},H_{A_1\cap A_3},H_{A_1\cup A_3}\}$ or
		\item $\{H_{A_1},H_{A_3},H_{(A_1\cap A_3)\cup X},H_{(A_1\cup A_3)\setminus X}\}$.
	\end{enumerate}
	In each case, we assume that the last element in each set is the largest with respect to the ordering $<$.
\end{prop}
Before proving this proposition, we give examples for each such type of circuit of cardinality four.

\begin{exmp}\label{ex:circuits}
	Consider the following families in the arrangement $\A_4$ corresponding to the cases of \Cref{prop:circuits}.
	\begin{enumerate}[(i)]
		\item The family $\{H_{\{1,2\}},H_{\{1,3\}},H_{\{2,3\}},H_{\{1,2,3\}}\}$ is a circuit of $\A_4$ since there is the relation $\chi_{\{1,2\}}+\chi_{\{1,3\}} + \chi_{\{2,3\}} = 2 \chi_{\{1,2,3\}}$.
		\item The family $\{H_{\{1,2\}},H_{\{1,3\}},H_{\{1\}},H_{\{2,3\}}\}$ is a circuit of $\A_4$ since there is the relation $\chi_{\{1,2\}}+\chi_{\{1,3\}} = 2\chi_{\{1\}} +  \chi_{\{2,3\}}$.
		\item The family $\{H_{\{1,2\}},H_{\{1,3\}},H_{\{1\}},H_{\{12,3\}}\}$ is a circuit of $\A_4$ since there is the relation $\chi_{\{1,2\}}+\chi_{\{1,3\}} = \chi_{\{1\}} +  \chi_{\{1,2,3\}}$.
		\item Setting $A_1\coloneqq \{2,4\},A_3\coloneqq \{1,3,4\}$ and $X\coloneqq \{1\}$ yields the family $\{H_{\{2,4\}},H_{\{1,3,4\}},\allowbreak H_{\{1,4\}},H_{\{2,3,4\}}\}$. This is a circuit of $\A_4$ since there is the relation $\chi_{\{2,4\}}+\chi_{\{1,3,4\}} = \chi_{\{1,4\}} +  \chi_{\{2,3,4\}}$.
	\end{enumerate}
\end{exmp}

\begin{proof}[Proof of \Cref{prop:circuits}]
	Generalizing the relations given in \Cref{ex:circuits} to arbitrary sets $A_1,A_3,X$ satisfying the conditions in \Cref{eq:sets} shows that these given families are indeed families of four different subsets of $\left[n\right]$ which form relevant circuits in $\A_n$.
	
	Conversely, let $\{A_1,\dots,A_4\}$ be a family of subsets corresponding to a relevant circuit in $\A_n$ with $A_i\neq A_j$ for any $i\neq j$, $A_i\cap A_j\neq \emptyset$ for $1\le i,j\le 3$ and $A_4$ is the maximal element in the ordering $<$.
	Since the hyperplanes form a circuit in $\A_n$ there is a relation $\sum_{i=1}^4\lambda_i\chi_{A_i}=0$ for some $\lambda_i\in\Z$ for $1\le i \le 4$.
	The coefficients $\lambda_i$ need to be non-zero since the circuit would otherwise satisfy a dependency of cardinality less than four.
	
	Using the symmetry of the sets $A_1,\dots,A_3$ it suffices to consider the two cases $\lambda_1,\lambda_2,\lambda_3>0$ and $\lambda_4<0$ or $\lambda_1,\lambda_3>0$ and $\lambda_2,\lambda_4<0$.
	Note, that the case $\lambda_1>0$ and $\lambda_2,\lambda_3,\lambda_4<0$ cannot occur since $A_4$ is the maximal element.
	
	\begin{description}
		\item[Case 1: $\lambda_1,\lambda_2,\lambda_3>0$ and $\lambda_4<0$]
		In this case, the relation implies $A_1\cup A_2\cup A_3=A_4$.
		Since the sets $A_1,A_2,A_3$ are by assumption pairwise intersecting every element in $A_4$ is contained in at least two of the sets $A_1,A_2,A_3$.
		Not all elements of $A_4$ appear in all of the sets $A_1,A_2,A_3$ since otherwise these four sets would all be equal.
		Hence, the relation then implies that every element in $A_4$ is contained in exactly two of the sets $A_1,A_2,A_3$ which means that we can without loss of generality assume $A_2=A_1\triangle A_3$.
		Therefore, the family is a circuit of type $(i)$.
		
		\item[Case 2: $\lambda_1,\lambda_3>0$ and $\lambda_2,\lambda_4<0$]
		Analogously to the first case, the relation  now yields $A_1\cup A_3= A_2\cup A_4$.
		Hence, the maximality of $A_4$ yields $A_1\not \subseteq A_3$ and $A_1\not \supseteq A_3$.
		Thus, the elements in $A_1\cup A_3$ are partitioned into the three blocks $A_1\setminus A_3, A_3\setminus A_1$ and $A_1\cap A_3$ appearing with positive coefficients $\lambda_1,\lambda_3$ and $\lambda_1+\lambda_3$ respectively in the relation.
		
		Assume there is an element $a\in (A_1\cup A_3)\setminus A_2$.
		Then, $a\in A_4$ which implies $\lambda_4=\lambda_1+\lambda_3$ since $a\not \in A_2$.
		This yields $A_4 \subseteq A_1\cap A_3$ which contradicts the maximality of $A_4$.
		Therefore, we must have $A_1\cap A_3\subseteq A_2$ and it suffices to consider the following two subcases:
		\begin{description}
			\item[Case 2.1: $A_1 \cap A_3=A_2$]
			Then we obtain $A_1\triangle A_3 \subseteq A_4$.
			Since the positive coefficients in the relation are constant on the block $A_1 \cap A_3$ we must have either $A_1\triangle A_3 = A_4$ or  $A_1\cup A_3 = A_4$.
			The former case yields a circuit of type $(ii)$ and the latter one of type $(iii)$ as described in the statement of \Cref{prop:circuits}.
			
			\item[Case 2.2: $A_1 \cap A_3\subsetneq A_2$]
			Assume $(A_1 \cap A_3)\cup X = A_2$ for some non-empty subset $X\subseteq A_1\triangle A_3$.
			Now, we must have $A_4 \supseteq (A_1\triangle A_3)\setminus X$ since $A_1\cup A_3= A_2\cup A_4$.
			Since $X\subseteq A_1\triangle A_3$, the coefficient $\lambda_2$ can be at most $\lambda_1$ or $\lambda_3$.
			However, the positive coefficient of the elements in $A_1 \cap A_3$ is $\lambda_1 +\lambda_3$.
			Hence, $A_4 \supseteq (A_1\cap A_3)$.
			So in total $A_4 \supseteq (A_1\cup A_3)\setminus X$.
			Since the positive coefficients of the elements in $A_1 \cap A_3$ and $A_1 \triangle A_3$ are different we must have $A_4\cap X = \emptyset$.
			Therefore, $A_4 = (A_1\cup A_3)\setminus X$ and the circuit is of type $(iv)$.
			\qedhere
		\end{description}
	\end{description}
\end{proof}

\Cref{prop:circuits} implies that all broken circuits of $\A_n$ of cardinality three are of the form $\{H_{A_1},H_{A_3},H_{A_1\triangle A_3}\}$ or $\{H_{A_1},H_{A_3},H_{(A_1\cap A_3)\cup X}\}$ for $A_1,A_3,X\subseteq \left[ n \right] $ with $A_1\cap A_3\neq \emptyset$, $A_1\not \subseteq A_3$, $A_1\not \supseteq A_3$ and $X\subseteq A_1\triangle A_3$.
The former ones correspond to circuits of type~$(i)$ with the relation $\chi_{\{A_1\}}+\chi_{\{A_2\}} + \chi_{\{A_3\}} = 2 \chi_{\{A_4\}}$.
We call them \emph{tetrahedron circuits} since they exhibit a tetrahedron if we regard the elements as vertices of the $n$-dimensional hypercube.

The latter broken circuits might not stem from a unique circuit of cardinality four.
We can however fix a bijection between these broken circuits and the circuits of type $(iii)$ and~$(iv)$ in \Cref{prop:circuits}.
These all satisfy the relation $\chi_{\{A_1\}} + \chi_{\{A_3\}} = \chi_{\{A_2\}}+ \chi_{\{A_4\}}$.
The characteristic functions of these circuits viewed in the $n$-dimensional hypercube form rectangles which is why we call these circuit \emph{rectangle circuits} in the following.

Using again \Cref{thm:betti_bc} to determine $b_3(\A_n)$ we will therefore start from \Cref{thm:intersecting_families} and subtract the number of tetrahedron and rectangle circuits which give broken circuits of cardinality three by removing the largest element in each circuit.
Note that a broken circuit can not stem from a tetrahedron and rectangle circuit simultaneously since it can not satisfy a tetrahedron and a rectangle relation at the same time.

\begin{prop}\label{prop:tetrahedra}
	For any $n\ge 1$ there are $S(n+1,4)$ tetrahedron circuits in $\A_n$.
\end{prop}
\begin{proof}
	Let $P_1,P_2,P_3,P_4$ be any partition of $\left[ n+1\right]$ where we label the parts so that  $n+1\in P_4$.
	Set $A_4 \coloneqq \left[ n+1\right] \setminus P_4$ and $A_i \coloneqq A_4 \setminus P_i$ for $1\le i \le 3$.
	
	We claim that the hyperplanes corresponding to $A_1,\dots,A_4$ form a tetrahedron circuit in~$\A_n$.
	By definition we have $P_k=A_i\cap A_j$ for any possible ordering $\{k,i,j\}=\{1,2,3\}$ and $A_i\subset A_4$ for all $1\le i \le 3$
	Hence, the family $A_1,\dots,A_4$ is pairwise intersecting, i.e. $A_i\cap A_j \neq \emptyset$ for all $i\neq j$.
	Next, consider $l\in A_4$ such that $l\in P_i$ for some $1\le i \le 3$ and set $\{j,k\}\coloneqq \{1,2,3\}\setminus \{i\}$.
	Then, we conclude that $l\in A_j, l\in A_k$ and $l\not\in A_i$ which implies that  $A_1,\dots,A_4$ corresponds to a tetrahedron circuit.
	
	Conversely, given the subsets $A_1,\dots, A_4$ of $\left[ n\right]$ corresponding to a tetrahedron circuit with largest subset $A_4$ we can define a partition of $\left[ n+1\right]$ by setting $P_4 \coloneqq \left[ n+1\right]\setminus A_4$ and $P_i \coloneqq A_4 \setminus A_i$ for $1\le i \le 3$.
	We claim this defines a partition of $\left[ n+1\right]$.
	By definition we have $P_i\cap P_4=\emptyset$ for all $1\le i\le 3$.
	The assumption of $A_1,\dots, A_4$ corresponding to a tetrahedron circuit implies that every $l\in A_4$ is contained in exactly two subsets $A_k,A_j$ for some $1\le k <j\le3$.
	This implies that every $l\in A_4$ is contained in exactly one block $P_i$ which proves that $P_1,\dots,P_4$ is a partition of $\left[ n+1\right]$.
	
	Since these two constructions are inverse to each other the claim follows.
\end{proof}

To count the rectangle circuits we construct corresponding tuples which will be easier to count.
Throughout the subsequent discussion we regard the indices cyclically, i.e.\ given any family of sets $X_1\dots X_n$ we set $X_0\coloneqq X_n$ and $X_{n+1}\coloneqq X_1$.

\begin{prop}\label{prop:vertices_sides}
	Let $(A_1,\dots,A_4)$ be a family of distinct and non-empty subsets of $\left[ n\right]$ forming a relevant rectangle circuit, i.e.\ $\chi_{A_1}+\chi_{A_3}= \chi_{A_2}+ \chi_{A_4}$ and $A_i\cap A_j \neq \emptyset$ for $1\le i < j \le 3$ with maximal element $A_4$.
	Then ,we define its \emph{midpoint} as $M\coloneqq \bigcap_{i=1}^4 A_i$ and the \emph{sides} of the rectangle as $S_i\coloneqq (A_i \cap A_{i+1}) \setminus M$ for $1\le i \le 4$. 
	
	In this case, the tuple $(S_1,\dots,S_4,M)$ satisfies
	\begin{enumerate}
		\item[$(S1)$] $S_i\cap S_j = \emptyset$ for all $i\neq j$ and in particular $S_i\neq S_j$ for all $i\neq j$,
		\item[$(S2)$]  $M\cap S_i=\emptyset$ for all $1\le i \le 4$,
		\item[$(S3)$] $M\neq \emptyset$, and
		\item[$(S4)$] at most one of two opposite sides are empty.
	\end{enumerate}
	We will call a tuple $(S_1,\dots,S_4,M)$ satisfying $(S1)$ to $(S4)$ a \emph{side-midpoint tuple}.
\end{prop}

\begin{exmp}
	 \Cref{fig:rect_circuits} depicts the general case of a rectangle circuit together with its  corresponding side-midpoint tuples as defined in \Cref{prop:vertices_sides} and two examples in~$\A_5$.
	
	\begin{figure}[ht!]
				\begin{subfigure}[b]{0.32\linewidth}
			\centering
			\includegraphics[scale=.45]{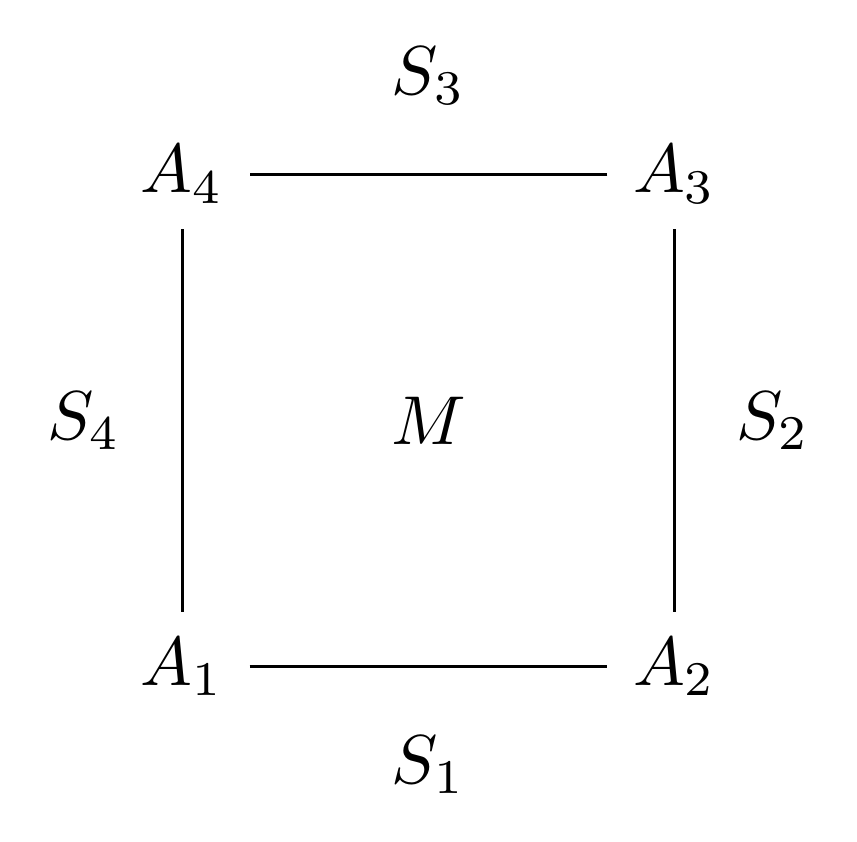}
			\label{fig:ex_general}
		\end{subfigure}
		\begin{subfigure}[b]{0.32\linewidth}
			\centering
			\includegraphics[scale=.45]{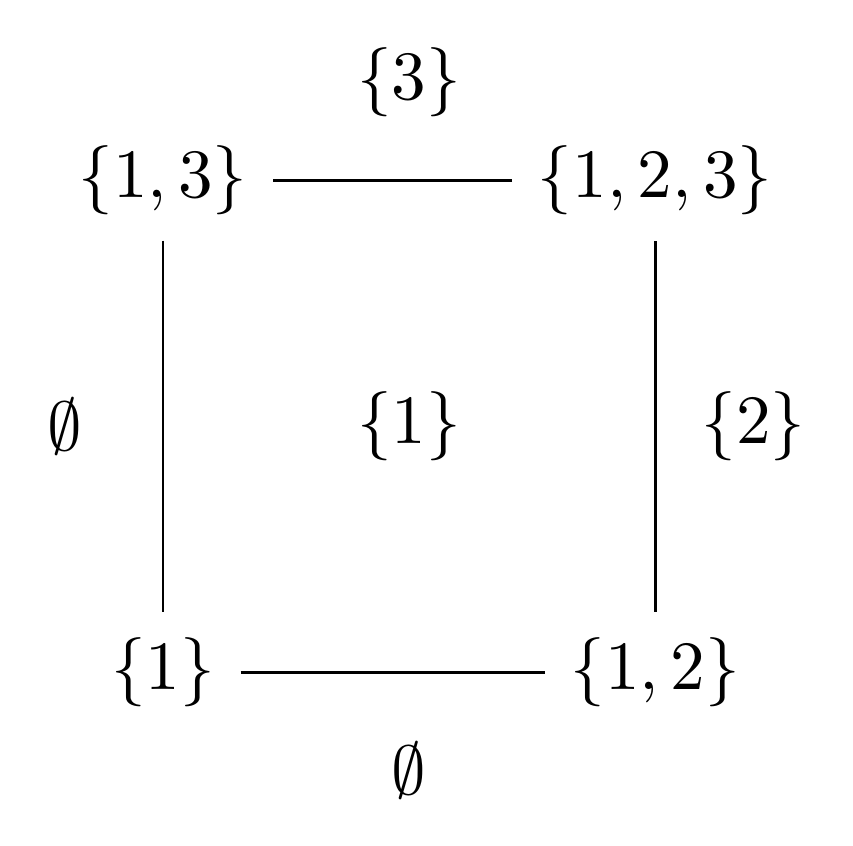}
			\label{fig:ex1}
		\end{subfigure}
	\begin{subfigure}[b]{0.32\linewidth}
		\centering
		\includegraphics[scale=.45]{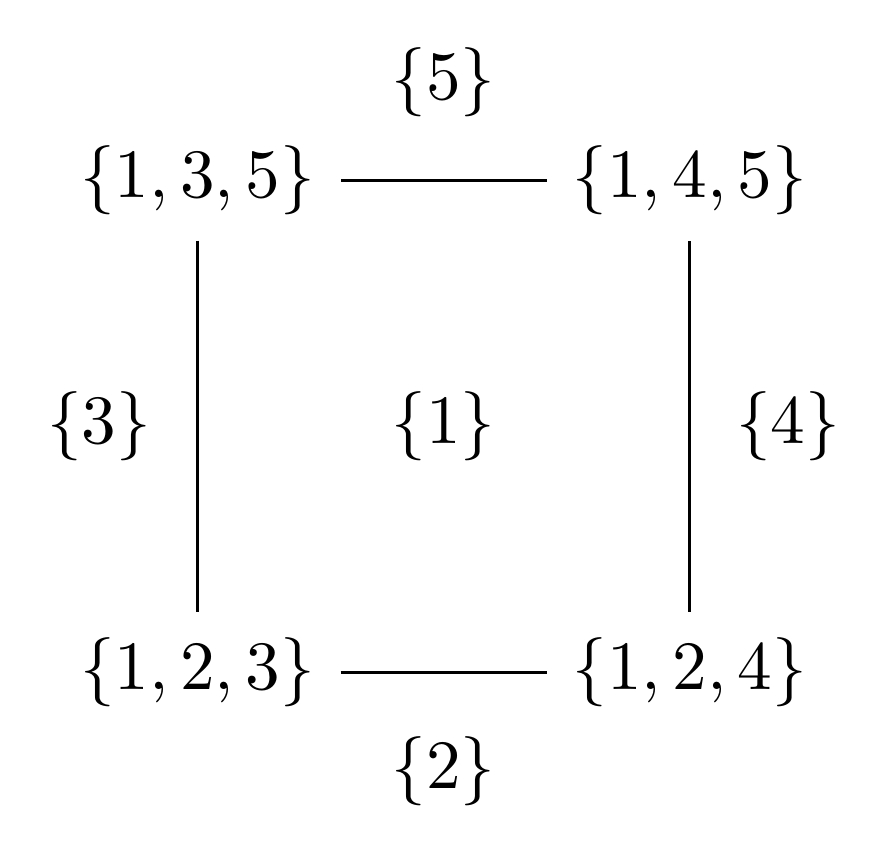}
		\label{fig:rex3}
	\end{subfigure}
		\caption{Three examples of rectangle circuits together with their side-midpoint tuples.}
		\label{fig:rect_circuits}
	\end{figure}

\end{exmp}
\begin{proof}[Proof of \Cref{prop:vertices_sides}]
	To prove $(S1)$ assume for a contradiction $a\in S_i\cap S_j$.
	Without loss of generality we can assume $a\in S_1\cap S_2$.
	By definition this yields $a\in A_1,A_2,A_3$ but $a\not\in M$.
	Thus $a\not \in A_4$.
	This contradicts the relation $\chi_{\{A_1\}} + \chi_{\{A_3\}} = \chi_{\{A_2\}}+ \chi_{\{A_4\}}$ in the element $a$.
	Thus, $S_i\cap S_j=\emptyset$ for all $i\neq j$.
	
	The sides $S_i$ are defined as $S_i\coloneqq (A_i \cap A_{i+1}) \setminus M$.
	This immediately implies property $(S2)$ namely $S_i\cap M=\emptyset$.
	
	By assumption, we have $A_1\cap A_3 \neq \emptyset$.
	The relation $\chi_{\{A_1\}} + \chi_{\{A_3\}} = \chi_{\{A_2\}}+ \chi_{\{A_4\}}$ then yields $A_1\cap A_3=A_2\cap A_4$.
	Therefore, $A_1\cap A_3=M\neq \emptyset$ which proves property $(S3)$.
	
	Lastly, assume without loss of generality $S_1=S_3=\emptyset$.
	This implies $A_1=M\cup S_4 \cup \widetilde{A_1}$ for some $\widetilde{A_1}\subseteq \left[n \right]$ disjoint from $M$ and $S_4$.
	This yields $\widetilde{A_1}\cap A_4 = \emptyset$ since any intersection of these sets disjoint from $M$ would be contained in $S_4$.
	Hence using the fact $A_1\cup A_3= A_2\cup A_4$, we obtain $\widetilde{A_1}\subseteq A_2$.
	Thus, \[\widetilde{A_1}\subseteq (A_1\cup A_2)\setminus M= S_1=\emptyset.\] 
	Hence, $\widetilde{A_1}=\emptyset $ and $A_1=M\cup S_4$.
	Analogously, we obtain $A_4=M\cup S_4$ which contradicts $A_1\neq A_4$.
\end{proof}

The next proposition shows that we can obtain a rectangle circuit from a side-midpoint tuple:

\begin{prop}\label{prop:sides_vertices}
	Let $(S_1,\dots,S_4,M)$ be a side-midpoint tuple.
	Set $A_i \coloneqq M \cup S_{i-1} \cup S_{i}$.
	Then, the family $(A_1,\dots,A_4)$ corresponds to a relevant rectangular circuit which means it satisfies
	\begin{enumerate}
		\item[$(C1)$] $A_i\neq A_j$ for all $i\neq j$,
		\item[$(C2)$] $A_i\neq\emptyset$ for all $1\le i \le 4$,
		\item[$(C3)$] $A_i\cap A_j\neq \emptyset$ for all $i\neq j$ and
		\item[$(C4)$] it forms a rectangle circuit, i.e.\ $\chi_{A_1}+\chi_{A_3}= \chi_{A_2}+ \chi_{A_4}$.
	\end{enumerate}
\end{prop}
\begin{proof}
	Assume $A_1=A_2$.
	This implies $M\cup S_{4}\cup S_1=M\cup S_{1}\cup S_2$.
	Hence, $S_4=S_2$.
	By assumption $(S1)$ these sets are disjoint which yields $S_4=S_2=\emptyset$.
	This contradicts the assumption $(S_4)$ that at most one of two opposite sets is empty.
	Now assume $A_1=A_3$.
	This implies $M\cup S_{4}\cup S_1=M\cup S_{2}\cup S_3$.
	Thus, we have two partitions of the same set by pairwise disjoint sets which can not all be empty which is impossible.
	Thus we have without loss of generality proven $(C1)$.
	
	By assumption we have $M\neq \emptyset$.
	Our construction of the sets $A_i$ yields  $M\subseteq A_i$ for all $1\le i\le 4$.
	This immediately implies $A_i\neq \emptyset$ for all $1\le i\le 4$ and $A_i\cap A_j\neq \emptyset$ for all $i\neq j$.
	Hence, properties $(C2)$ and $(C3)$ hold.
	
	Lastly, we have by construction of the sets $A_i$ and due to the fact that the sets $S_1,\dots,S_4,M$ are pairwise disjoint 
	\[
		\chi_{A_1}+\chi_{A_3}=\chi_M+\sum_{i=0}^4\chi_{S_i}=\chi_{A_2}+ \chi_{A_4}.\qedhere
	\]
\end{proof}

\begin{prop}\label{prop:inverses}
	The constructions defined in \Cref{prop:vertices_sides,prop:sides_vertices} are inverse to each other.
\end{prop}
\begin{proof}
	Let $(A_1,\dots,A_4)$ be the vertices of a relevant rectangle circuit satisfying $(C1)$ to $(C4)$.
	This yields by \Cref{prop:vertices_sides} the side-midpoint tuple with midpoint $M_A
	\coloneqq \bigcup_{i=1}^4 A_i$ and sides $(A_{i-1}\cap A_i)\setminus M_A$.
	Fix some $1\le i \le 4$.
	The relation in property $(C4)$ then implies $A_i \subseteq A_{i-1}\cup A_{i+1}$.
	Hence, we obtain $A_i = (A_{i-1}\cup A_i) \cup (A_{i}\cup A_{i+1})$.
	This yields,
	\[
		A_i = M_A \cup ((A_{i-1}\cup A_i)\setminus M_A) \cup ((A_{i}\cup A_{i+1})\setminus M_A).
	\]
	Thus, the vertices $A_i$ equal the resulting vertices from the construction in \Cref{prop:sides_vertices}.

	Conversely, let $(S_1,\dots,S_4,M)$ be a side-midpoint tuple.
	This yields by \Cref{prop:sides_vertices} the vertices of a rectangle circuit $M\cup S_{i-1}\cup S_i$ for $1\le i \le 4$.
	Since the sets $S_1,\dots,S_4,M$ are pairwise disjoint the construction of \Cref{prop:vertices_sides} applied to these vertices yields the side-midpoint tuple $(S_1,\dots,S_4,M)$.
\end{proof}

In total we have established a bijection between relevant rectangle circuits and side-midpoint tuples.
The former correspond to broken circuits of $\A_n$ of the form $\{H_{A_1},H_{A_3},\allowbreak H_{(A_1\cap A_3)\cup X}\}$ for $A_1,A_3,X\subseteq \left[ n \right] $ with $A_1\cap A_3\neq \emptyset$, $A_1\not \subseteq A_3$, $A_1\not \supseteq A_3$ and $X\subseteq A_1\triangle A_3$.
We are now able to determine the number of these broken circuits by counting side-midpoint tuples.

\begin{prop}\label{theo:rectangle_circuits}
	For any $n\ge 1$ there are $3S(n+1,4)+12S(n+1,5) + 15 S(n+1,6)$ side-midpoint tuples in $\left[n\right]$.
	This number equals the relevant rectangle circuits in $\A_n$.
\end{prop}
\begin{proof}
	We split up the side-midpoint tuples in $\left[n\right]$ into three cases depending on how many sides are empty.
	Since at most one of two opposite sides can be empty these cover all side-midpoint tuples.
	\begin{description}
		\item[Case 1: Two adjacent sides are empty.] 
		Say $S_1=S_2=\emptyset$.
		In this case, we need to count partitions of a subset of $\left[n\right]$ into three blocks, one for each of the sets $S_3,S_4$ and $M$.
		The sets $S_3$ and $S_4$ are symmetric and we can choose any of the three blocks for the distinguished set $M$.
		Therefore, we obtain $3S(n+1,4)$ side-midpoint tuples in this case.
		\item[Case 2: Exactly one side is empty.] 
		Say $S_1=\emptyset$.
		In this case, we need to count partitions of a subset of $\left[n\right]$ into four blocks, one for each of the sets $S_2,S_3,S_4$ and $M$.
		There are $S(n+1,5) $ such partitions.
		We can choose any of the four blocks as the distinguished midpoint $M$.
		The remaining three blocks can be assigned to the sets $S_2,S_3,S_4$ in exactly three non-equivalent ways.
		These choices correspond to the identity permutations and the two transposition $(1\; 2)$ and $(2\;3)$ in $\mathfrak{S}_3$
		Therefore there are in total $12S(n+1,5)$ side-midpoint tuples in this case.
		\item[Case 3: All sides are non-empty.]
		This case works almost analogously to Case 2.
		This time we need to count partitions of a subset of $\left[n\right]$ into five blocks, one for each of the sets $S_1,S_2,S_3,S_4$ and $M$.
		There are $S(n+1,6) $ such partitions.
		We can choose any of the blocks as the midpoint.
		Subsequently, we can  fix $S_1$ as the first free block without any choices due to the symmetry of the sets $S_1,\dots,S_4$.
		As in Case 2 there are now three choices for the assignment of the last three sets.
		In total we obtain $15S(n+1,6)$ side-midpoint tuples without any empty sides. \qedhere
	\end{description}
\end{proof}

Putting the above statements together we can prove the announced formula for $b_3(\A_n)$:
\begin{proof}[Proof of~\Cref{thm:betti} $(ii)$]
	By \Cref{thm:betti_bc}, the Betti number $b_3(\A_n)$ equals the number of intersecting families of cardinality three minus the number of broken circuits of cardinality three.
	Hence, we can compute $b_3(\A_n)$ using~\Cref{eq:intersection_family_stirling} in~\Cref{thm:intersecting_families} subtracted by the number of tetrahedron and rectangle circuits computed in \Cref{prop:tetrahedra} and \Cref{theo:rectangle_circuits}.
	Thus, we obtain
	\[
		b_3(\A_n) = 9S(n+1,4)+80S(n+1,5) + 345 S(n+1,6) + 840 S(n+1,7) +840S(n+1,8).
	\]
	Expanding this equation via the formula for the Stirling numbers in~\Cref{eq:stirling} yields
	\[
	b_3(\A_n) = \frac{1}{4!} (4\cdot8^n -15\cdot 6^n +15\cdot 5^n - 14 \cdot 4^n + 18 \cdot 3^n - 7\cdot 2^n-1).\qedhere
	\]
\end{proof}

\bibliographystyle{myalpha}
\bibliography{matroid.bib}

\end{document}